\documentclass[a4paper,10pt]{article}

\usepackage[ascii]{inputenc}
\usepackage{amsmath,amsfonts,amssymb,amsthm}

\newtheoremstyle{theorem}{5pt}{5pt}{\itshape}{}{\bfseries}{.}{.5em}{}
\newtheorem{theorem}{Theorem}
\newtheorem{lemma}[theorem]{Lemma}
\newtheorem{proposition}[theorem]{Proposition}
\newtheorem{remark}{Remark}

\theoremstyle{definition}
\newtheorem*{definition*}{Definition}

\usepackage{titlesec}
\titleformat{\section}
{\Large\bfseries}{\thesection}{1em}{\large}
\titlespacing*{\section}{0pt}{3.5ex plus 1ex minus .2ex}{2.3ex plus .2ex}

\titleformat{\subsection}
{\normalfont\bfseries}{\thesubsection}{1em}{\normalsize}
\titlespacing*{\section}{0pt}{3.5ex plus 1ex minus .2ex}{2.3ex plus .2ex}

\allowdisplaybreaks[3]

\begin{document}

\title{Discreteness of Transmission Eigenvalues for  Higher-Order Main Terms and Perturbations}
\author{Andoni Garc\'{\i}a\footnote{Department of Mathematics, University of the Basque Country, 48080, Bilbao, Spain}{ },
Esa V\!. Vesalainen\footnote{Aalto University, Department of Mathematics and Systems Analysis, P.O. Box 11100, FI-00076 Aalto, FINLAND}{ },
Miren Zubeldia\footnote{BCAM - Basque Center for Applied Mathematics, Mazarredo, 14 E48009 Bilbao, Basque Country - Spain and Department of Mathematics and Statistics, University of Helsinki, Finland}}
\date{}
\maketitle

\begin{abstract}
In this paper we extend Sylvester's approach via upper triangular compact operators to establish the discreteness of transmission eigenvalues for higher-order main terms and higher-order perturbations. The coefficients of the perturbations must be sufficiently smooth and the coefficients of the higher-order terms of the perturbation must vanish in a neighbourhood of the boundary of the underlying domain. The zeroeth order term must satisfy a suitable coercivity condition in a neighbourhood of the boundary.
\end{abstract}

\section{Introduction}

Let $P$ be a formally self-adjoint elliptic constant coefficient partial differential operator in $\mathbb R^n$ bounded from below of order $k\in\mathbb Z_+$, and let $\Omega\subset\mathbb R^n$ be a bounded non-empty open set. We will consider the following interior transmission eigenvalue problem associated to $P$,
\begin{equation*}\label{eq:ITP}
\begin{cases}
(P+Q-\lambda)v=0 & \text{in $\Omega$},\\
(P-\lambda)w=0 & \text{in $\Omega$},\\
v-w\in H^k_0(\Omega).
\end{cases}
\end{equation*}
Here $H^k_0(\Omega)$ denotes the Sobolev space defined as the closure of $C^{\infty}_{\mathrm c}(\Omega)$ in the Sobolev norm $H^k(\Omega)$.
 We say that $\lambda\in\mathbb C$ is a transmission eigenvalue if there exists a non-zero pair of $L^2$-functions $(v,w)$ solving the above system. The multiplicity of a transmission eigenvalue is the dimension of the space of solutions. 
In the equations above, the perturbation $Q$ will be a partial differential operator of order smaller than $k$.

The problem of transmission eigenvalues was introduced in \cite{Colton--Monk, Kirsch} in connection with an inverse scattering problem for the reduced wave equation. The discreteness of the set of transmission eigenvalues was among the first general results obtained \cite{Colton--Kirsch--Paivarinta}. The original motivation for studying them was largely derived from the fact that some qualitative methods of inverse scattering theory, namely the linearization method \cite{Colton--Kirsch} and the factorization method \cite{Kirsch2}, require using energies (or wavenumbers or frequencies) for which every incident wave scatters non-trivially. Energies which do not satisfy this condition turn out to be transmission eigenvalues, and so the discreteness of transmission eigenvalues implies that of non-scattering energies.

For the existence of transmission eigenvalues the first general result was proved in \cite{Paivarinta--Sylvester}, and the existence of infinitely many transmission eigenvalues was established in \cite{Cakoni--Gintides--Haddar} soon after. This gave more impetus to the study of the topic, as transmission eigenvalues provide a new potential avenue for deriving information about a scatterer. In particular, the knowledge of interior transmission eigenvalues can be used in determining a radial scatterer \cite{Maclaughlin--Polyakov, Maclaughlin--Polyakov--Sacks, Colton--Kirsch--Paivarinta}, and for non-radial scatterers, also some information about the scatterer can be derived~\cite{Cakoni--Colton--Monk}.

Most of the work on transmission eigenvalues has so far been for second-order main terms and zeroeth order perturbations.
The series of papers \cite{Hitrik--Krupchyk--Ola--Paivarinta1, Hitrik--Krupchyk--Ola--Paivarinta2, Hitrik--Krupchyk--Ola--Paivarinta3} were the first to consider transmission eigenvalues for higher-order main terms. More concretely, they considered the case of general constant coefficient operators $P$, and for these a well developed scattering theory is available~\cite{Hormander}.

For more information and references on transmission eigenvalues we recommend \cite{Cakoni--Haddar1, Cakoni--Haddar2}.

\subsection{Why higher order main terms and perturbations?}

Our motivation for considering more general higher-order main terms and higher-order perturbations is twofold. First, one would naturally like to strive for as much generality as is reasonably possible. The works \cite{Hitrik--Krupchyk--Ola--Paivarinta1, Hitrik--Krupchyk--Ola--Paivarinta2, Hitrik--Krupchyk--Ola--Paivarinta3} demonstrate how the basic features of the theory of transmission eigenvalues pleasantly carry through to higher-order main terms. Furthermore, as far as we know, higher-order perturbations have not been considered in this connection before.

Second, higher-order operators and perturbations appear in many places in both mathematics and applications. Let us mention only a few examples: fourth order equations and second order perturbations in plate tectonics and more generally thin elastic plates in mechanics \cite{Villaggio}, equations of quantum field theory \cite{Esposito--Kamenshchik}, the Paneitz--Branson operator of conformal geometry \cite{Branson}, and the characterizing equation for boundary values of polyharmonic functions in the unit ball of $\mathbb C^2$ \cite{Bedford, Bedford--Federbush}. Also sets of generators of centers of the algebras of invariant differential operators in many homogeneous spaces, whose joint eigenfunctions are the central objects of harmonic analysis in homogeneous spaces \cite{Helgason}, often include higher-order operators. One notable instance of this are spaces such as $\mathrm{SL}(n,\mathbb R)/SO(n,\mathbb R)$ which are of great importance in number theory \cite{Goldfeld}. Last, but definitely not least, already for second-order main terms first-order perturbations appear when magnetic potentials are present in Schr\"odinger scattering.

\subsection{Higher-order main terms with higher-order perturbations}\label{Esa}

Arguably the most common approach to dealing with transmission eigenvalues would involve considering certain quadratic forms involving the inverse of the perturbation. With such an approach this inverse would pose obvious challenges for higher-order perturbations. The recent novel approach of Sylvester \cite{Sylvester} to establishing discreteness of transmission eigenvalues, which gives the most general discreteness result for $P=-\Delta$ and related divergence form main terms to date, instead turns out to be rather more amenable to such generalizations.

What we consider are perturbations of order lower than $k$ in which the positive order terms have coefficients with enough smoothness and vanish in a neighbourhood of $\partial\Omega$, and in which the zeroeth order term satisfies a suitable coercivity condition. More precisely, in this case $Q= W + \lambda^\nu V$ where $V\in L^\infty(\Omega)$ is the complex-valued zeroeth order term, $\nu\in\left\{0,1\right\}$ and $W$ is a partial differential operator of the form
$$
W = \sum_{1 \leqslant |\alpha| \leqslant e} W_\alpha \partial^\alpha 
$$ 
with $e\in \{0, 1, \ldots, k-1\}$ and the complex-valued coefficients $W_\alpha$ are smooth enough for each multi-index $\alpha$ and vanish in a neighbourhood of $\partial\Omega$. We are planning to relax this vanishing condition for the magnetic Schr\"odinger operator elsewhere.

This approach requires some a priori estimates which we derive in the spirit of \cite{Robbiano} using parameter-dependent pseudodifferential calculus. We emphasize here that no smoothness is required from $\partial\Omega$, and $V$ only needs to satisfy the coercivity condition near $\partial\Omega$; otherwise $V$ can be an arbitrary complex valued $L^\infty$-function.

The two cases $\nu=0$ and $\nu=1$ could be called Schr\"odinger and Helmholtz cases, respectively. Perturbations with $\nu=0$ appear in quantum scattering, whereas potentials with $\nu=1$ appear in electromagnetic and acoustic scattering. Although the proofs are mostly parallel for the two cases, the case $\nu=1$ is harder as in the end one needs to perform a perturbation argument with large $\lambda$ which is harder for the rather large perturbation $\lambda V$.

The case $\nu=1$ is more delicate in other ways as well: In the case where $W$ is present, the Helmholtz argument will involve $W/\lambda$, and this prevents us from excluding the possibility that the transmission eigenvalues accumulate to zero. Furthermore, for $\nu=1$, we need to invoke unique continuation and this imposes restrictions on the main term. Finally, in the case $\nu=1$, we can only treat one of the two coercivity conditions which appear in \cite{Sylvester}. We are planning to consider the other coercivity condition elsewhere.


\subsection*{Notation}

In various exponents, $\varepsilon$ will denote an arbitrarily and sufficiently small positive real number whose value will change from one occurrence to the next. The symbol $\mathbb C^\times$ means the set of non-zero complex numbers.

In the integrals where we do not specify the integration space we mean that we are integrating over the domain $\Omega$ with respect to the Lebesgue measure $dx$, i.e. $\int = \int_{\Omega} dx$. In addition, if we do not specify the domain in function spaces, it means that we are considering the domain $\Omega$. For instance, $L^2 = L^2(\Omega)$. Furthermore, we write  $\Vert \cdot \Vert$ for $\Vert \cdot \Vert_{L^2}$.

We use the standard asymptotic notations $\ll$, $\gg$ and $\asymp$. For complex-valued functions $f$ and $g$, defined in some set $X$, the notation $f\ll g$ means that there exists a constant $C\in\mathbb R_+$ so that $\left|f(x)\right|\leqslant C\left|g(x)\right|$ for all $x\in X$. The implicit constant $C$ is always allowed to depend on the dimension $n\in\mathbb Z_+$ of the ambient Euclidean space, on $\varepsilon$, and on the domain $\Omega$, on the order $k\in\mathbb Z_+$ of the main terms, and on the potentials and coefficients $V$ and $W_\alpha$ appearing in the interior transmission problem, and on anything that has been explicitly fixed. If the implicit constant depends on some other objects $\alpha$, $\beta$, \dots, then we write $\ll_{\alpha,\beta,\dots}$ instead of $\ll$. The notation $g\gg f$ means the same as $f\ll g$. The notation $f\asymp g$ means that both $f\ll g$ and $f\gg g$.

\subsection*{Acknowledgements}

The first author received funding from the project MTM2011-24054 Ministerio de Ciencia y Tecnolog\'ia de Espa\~na.

The second author received funding from Finland's ministry of Education through the Doctoral Program in Inverse Problems, the Vilho, Yrj\"o and Kalle V\"ais\"al\"a Foundation, the Academy of Finland through the Finnish Centre of Excellence in Inverse Problems Research and the projects 283262, 276031 and 282938, and the European Research Council under the European Union's Seventh Framework Programme (FP/2007-2013) / ERC Grant Agreement n.\ 267700.

The third author is supported by the Basque Government through the grant POS-2014-1-43\ and also by the Basque Government through the BERC 2014-2017 program and by the Spanish Ministry of Economy and Competitiveness MINECO: BCAM Severo Ochoa accreditation SEV-2013-0323 and the project MTM2014-53145-P. She also received funding from CoE in Inverse Problems 2012-2017, University of Helsinki and UH/CoE project 79999103.

Finally, the second and third authors would like to acknowledge the generous support of the Henri Poincar\'e Institute, where part of this research was carried out during the thematic program in inverse problems in 2015.

\section{The main results}


Let $P$ be a formally self-adjoint elliptic constant coefficient partial differential operator of order $k\in\mathbb Z_+$ bounded from below, and let $\Omega$ be a bounded nonempty open set in $\mathbb R^n$. The symbol of $P$ is denoted by $P(\cdot)$.

The minimal extension of $P$ is the closed extension
\[P\colon H^k_0(\Omega)\longrightarrow L^2(\Omega),\]
where $H_0^k(\Omega)$ is the closure of $C_{\mathrm c}^\infty(\Omega)$ in the Sobolev norm $\left\|\cdot\right\|_{H^k(\Omega)}$.
The maximal extension of $P$ is the closed extension
\[P\colon H^k_P(\Omega)\longrightarrow L^2(\Omega),\]
where $H^k_P(\Omega)$ is the function space
\[H^k_P(\Omega)=\bigl\{u\in L^2(\Omega)\bigm| Pu\in L^2(\Omega)\bigr\}.\]
Furthermore, for $k\in\mathbb Z_+\cup\left\{0\right\}$ and $p\in\left[1,\infty\right]$, we denote by $W_{\mathrm c}^{k,p}(\Omega)$ those functions in $W^{k,p}(\Omega)$ which are supported in some compact subset of $\Omega$. Or equivalently, the functions in $W_{\mathrm c}^{k,p}(\Omega)$ are functions in $W^{k,p}(\Omega)$ each of which vanishes in some neighbourhood of $\partial\Omega$.

We recall here that the adjoint of the maximal extension of $P$ is the minimal extension of $P$, and vice versa. For more on minimal and maximal realizations we refer to Section 4.1 of \cite{Grubb}.

\medbreak
The following is the main theorem in the Schr\"odinger case.

\begin{theorem}\label{schrodinger-main-theorem}
Let $P$ be a formally self-adjoint elliptic constant coefficient partial differential operator of order $k\in\mathbb Z_+$ in $\mathbb R^n$ bounded from below, let $\Omega$ be a bounded nonempty subset of $\mathbb R^n$.
Let $V\in L^\infty(\Omega)$ be such that in some neighbourhood of $\partial\Omega$ it only takes values from some closed complex half-plane not containing zero. Furthermore, let $W_\alpha\in W^{\left|\alpha\right|,\infty}_{\mathrm c}(\Omega)$ be complex-valued for each multi-index $\alpha$ with $1\leqslant\left|\alpha\right|\leqslant e$, where $e\in\left\{0,1,\ldots,k-1\right\}$, and write $W$ for the partial differential operator
\[W=\sum_{1\leqslant\left|\alpha\right|\leqslant e}W_\alpha\partial^\alpha.\]
In the case $e=0$ we set $W=0$.
Let $\Lambda$ be the set of complex numbers $\lambda$ for which there exist functions $v,w\in L^2(\Omega)\setminus\left\{0\right\}$ with $v-w\in H^k_0(\Omega)$ solving the system
\[\left\{\!\!\begin{array}{l}
(P+W+V-\lambda)v=0,\\
(P-\lambda)w=0.
\end{array}\right.\]
Then the set $\Lambda$ is discrete, and each $\lambda\in\Lambda$ is of finite multiplicity.
\end{theorem}

\begin{remark}
Here $\lambda\in\Lambda$ is said to be of finite multiplicity if the space of pairs of solutions $\left(v,w\right)\in L^2(\Omega)\times L^2(\Omega)$ with $v-w\in H_0^k(\Omega)$ to the above system is finite dimensional.
\end{remark}

\begin{remark}Though $v$ and $w$ are a priori only in $L^2(\Omega)$, it follows from the second equation that $w\in H^k_P(\Omega)$, and from the condition for $v-w$ that also $v\in H^k_P(\Omega)$. Lemma \ref{estimates-for-w} below then guarantees that also $Ww$ is a well-defined $L^2$-function and again the condition for $v-w$ guarantees that the same is true for $Wv$.
\end{remark}

For the Helmholtz case we have the following theorem, which generalizes one of the two cases of Theorem 1.2 in \cite{Sylvester}. As the proof depends on unique continuation, we first define a subclass of elliptic operators following \cite{Wang}.
\begin{definition*}
An elliptic homogeneous constant coefficient partial differential operator $P_0$ is said to be of class $G$, if there exists a unit vector $\eta\in S^{n-1}$ such that for any $\xi\in\mathbb R^n$ with $P_0(\xi+i\eta)=0$ we have
\begin{itemize}\setlength{\itemsep}{0pt}
\item[(G1)] $\eta\cdot(\nabla_zP_0)(\xi+i\eta)\neq0$, where $\nabla_zP_0$ is the gradient $(\partial P_0/\partial z_1,\ldots,\partial P_0/\partial z_n)$ where $P_0=P_0(z_1,\ldots,z_n)$ is the obvious complex polynomial in $\mathbb C^n$, and
\item[(G2)] $H(P_0)(\xi+i\eta)\neq0$, where $H(P_0)(\cdot)$ is the determinant of the complex Hessian matrix of $P_0(\cdot)$.
\end{itemize}
The condition (G1) is called Calder\'on's simple characteristic condition, and the condition (G2) is called the curvature condition. Here $P_0(\cdot)$ denotes the symbol of $P_0$.
\end{definition*}

\begin{theorem}\label{helmholtz-main-theorem}
Let $P$ be a formally self-adjoint elliptic constant coefficient partial differential operator of order $k\in\mathbb Z_+$ in $\mathbb R^n$ bounded from below and with principal part of class $G$, let $\Omega$ be a bounded nonempty subset of $\mathbb R^n$.
Let $V\in L^\infty(\Omega)$ be such that $\Re V\geqslant-1+\delta$ in $\Omega$, and that in some neighbourhood of $\partial\Omega$ it only takes values with real parts in $\left[-1+\delta,-\delta\right]$, where $\delta\in\left]0,1\right[$.

Furthermore, let $W_\alpha\in W^{\left|\alpha\right|,\infty}_{\mathrm c}(\Omega)$ be complex-valued for each multi-index $\alpha$ with $1\leqslant\left|\alpha\right|\leqslant e$, where $e\in\left\{0,1,\ldots,k-1\right\}$, and write $W$ for the partial differential operator
\[W=\sum_{1\leqslant\left|\alpha\right|\leqslant e}W_\alpha\partial^\alpha.\]
In the case $e=0$ we set $W=0$. If $e>0$ or if $P$ is not homogeneous, we assume that $k<n/2$. Assume also that $\left\|\Im V\right\|_{L^\infty}$ is sufficiently small, depending on $\Omega$, $P$ and $W$.

Let $\Lambda$ be the set of non-zero complex numbers $\lambda$ for which there exist functions $v,w\in L^2(\Omega)\setminus\left\{0\right\}$ with $v-w\in H^k_0(\Omega)$ solving the system
\[\left\{\!\!\begin{array}{l}
(P+W+\lambda V-\lambda)v=0,\\
(P-\lambda)w=0.
\end{array}\right.\]
Then the set $\Lambda$ is a discrete subset of $\mathbb C^\times$, and each $\lambda\in\Lambda$ is of finite multiplicity. If $e=0$, then $\Lambda$ is a discrete subset of $\mathbb C$.
\end{theorem}

\begin{remark}
Actually, the proof offers more flexibility. For instance, we may let $W_\alpha$ and $V$ depend analytically on $\lambda\in D$, where $D$ is a connected neighbourhood of $\mathbb R$ in $\mathbb C$, to get discreteness and finite multiplicity in $D$ or $D\setminus\left\{0\right\}$, as appropriate. The coefficients $W_\alpha$ should vanish in a neighbourhood of the boundary which does not depend on $\lambda$, and their $W^{\left|\alpha\right|,\infty}$-norms should be uniformly bounded for large negative reals $\lambda$. In fact we could allow even a little growth such as $\left\|W_\alpha\right\|_{W^{\left|\alpha\right|,\infty}}\ll\left|\lambda\right|^{\delta}$ for small enough $\delta\in\mathbb R_+$. Similarly, in a compact set of $\Omega$, the potential $V$ can change analytically fairly arbitrarily as long as its $L^\infty$-norm stays uniformly bounded, or grows slowly enough, for large negative reals $\lambda$. Near the boundary, one could similarly allow analytic dependence as long as the details of the coercivity condition are uniform for large negative~$\lambda$.
\end{remark}

\begin{remark}
The conditions in the above theorems for $W_\alpha$ are formulated in terms of Sobolev spaces for simplicity. In fact, the proof only requires that $\partial^\beta W_\alpha\in L^\infty(\Omega)$ for multi-indices $\beta$ with $\beta_1\leqslant\alpha_1$, $\beta_2\leqslant\alpha_2$, \dots, $\beta_n\leqslant\alpha_n$. This condition comes from requiring that the adjoint $W^*$ has $L^\infty$-coefficients.
\end{remark}

\begin{remark}
The condition that $P$ should have a principal part of class $G$ and the condition on $m<n/2$ are only required to apply the weak unique continuation theorems, Theorems 1.3 and 1.4, from \cite{Wang}.
\end{remark}

\begin{remark}
In both main theorems the method of proof actually implies about real transmission eigenvalues that $\left]-\infty,-C\right]\cap\Lambda=\emptyset$ for some $C\in\mathbb R_+$ which depends on $\Omega$ and all the operators in question.
\end{remark}

\begin{remark}
Theorem \ref{helmholtz-main-theorem} does not exclude the possibility that the transmission eigenvalues accumulate to zero, unless $W=0$.
\end{remark}

\section{Proofs of Theorems \ref{schrodinger-main-theorem} and \ref{helmholtz-main-theorem}}

Our overall strategy for proving Theorems \ref{schrodinger-main-theorem} and \ref{helmholtz-main-theorem} is the same as in \cite{Sylvester}. Setting $u=v-w$, instead of the interior transmission problem we consider the equivalent problem for $u\in H^k_0(\Omega)$ and $w$:
\[
\begin{cases}
(P+Q-\lambda)u+Qw=0&\text{in $\Omega$,}\\
(P-\lambda)w=0&\text{in $\Omega$.}
\end{cases}
\]
Here $u$, in a sense, satisfies many ``boundary conditions'', but $w$ satisfies none. We first consider the ``Born approximation'' by simply striking out the term $Qu$ or $\lambda Vu$, depending on whether $\nu=0$ or $\nu=1$. Once the properties of the resolvent operators of the simpler case has been dealt with, the missing term can be brought back in with a perturbation argument. This is particularly simple in the Schr\"odinger case but requires more work in the Helmholtz situation.

The key goal is to prove that the Born approximation has well-defined resolvent operators when $\lambda$ is a negative real number and tends to $-\infty$, that their norms are well controlled, and that they are upper triangular compact. To do all this, we will first have to prove some a priori estimates. Following \cite{Robbiano}, we do so by using parameter-dependent pseudodifferential calculus.

We first review the pseudodifferential calculus needed, then establish the relevant a priori estimates. These estimates are then applied carefully with basic functional analysis to establish the resolvent operators and their good properties for the Born approximation. Finally, perturbation arguments allow us to move back to the original interior transmission problems, and the analogue of the analytic Fredholm theorem for upper triangular compact operators finishes the proof.

In this entire section, $P$ will be as in Theorems \ref{schrodinger-main-theorem} and \ref{helmholtz-main-theorem}: a fixed elliptic formally self-adjoint constant coefficient partial differential operator bounded from below and of fixed order $k\in\mathbb Z_+$. The symbol of $P$ will be denoted by the same letter, typically by writing $P(\cdot)$.

\subsection{Parameter-dependent pseudodifferential operators}

The proofs for the estimates for studying the Born approximation will make use of parameter-dependent pseudodifferential operators. In order to make the presentation more self-contained, we present here, with just enough generality, the relevant definitions and results. A standard reference for this topic is \cite{Shubin}.

Let $m\in\mathbb R$ and $k\in\mathbb Z_+$, and let $\Lambda\subseteq\left[1,\infty\right[$ be unbounded. For us, $\Lambda$ will be an interval $\left[C,\infty\right[$ for some, typically very large $C\in\mathbb R_+$. Here $m$ denotes the order of the pseudodifferential operator, and for us $k$ will be the order of the main terms in the interior transmission problem.

The symbol class $S^m_{1,0,k}(\mathbb R^n\times\mathbb R^n;\Lambda)$ consists of all those functions \[\sigma\colon\mathbb R^n\times\mathbb R^n\times\Lambda\longrightarrow\mathbb C,\] for which
$\sigma(\cdot,\cdot,\lambda)\in C^\infty(\mathbb R^n\times\mathbb R^n)$ for each $\lambda\in\Lambda$, and
\[\partial_\xi^\alpha\partial_x^\beta\sigma(x,\xi,\lambda)\ll_{\alpha,\beta,\sigma}
\left(\lambda^{1/k}+\left|\xi\right|\right)^{m-\left|\alpha\right|}\]
for all $x,\xi\in\mathbb R^n$ and $\lambda\in\Lambda$, and for all multi-indices $\alpha$ and $\beta$.

The operator $\mathrm{Op}(\sigma)$ corresponding to a symbol $\sigma\in S^m_{1,0,k}(\mathbb R^n\times\mathbb R^n;\Lambda)$ is defined for Schwartz test functions $\varphi\in S(\mathbb R^n)$ by the formula
\[\bigl(\mathrm{Op}(\sigma)\varphi\bigr)(x,\lambda)=\int\limits_{\mathbb R^n}e^{2\pi ix\cdot\xi}\,\sigma(x,\xi,\lambda)\,\widehat\varphi(\xi)\,\mathrm d\xi\]
for $x\in\mathbb R^n$ and $\lambda\in\Lambda$. All our pseudodifferential operators will depend on $\lambda$, but we will simplify the notation by writing $\sigma(x,\xi)$ and $\bigl(\mathrm{Op}(\sigma)\varphi\bigr)(x)$ without the lambdas.

If $m\leqslant0$, then $\mathrm{Op}(\sigma)$ extends to a bounded operator from $L^2(\mathbb R^n)$ to $L^2(\mathbb R^n)$, for each $\lambda\in\Lambda$, and more precisely, the operator norm has the pleasant upper bound
\[\bigl\|\mathrm{Op}(\sigma)\bigr\|\ll_\sigma\lambda^{m/k}.\]
We would like to specifically add that for $m=0$,
\[\bigl\|\mathrm{Op}(\sigma)\bigr\|_{H^s(\mathbb R^n)\longrightarrow H^s(\mathbb R^n)}\ll_{\sigma,s}1,\]
for all $s\in\mathbb R$, uniformly for $\lambda\in\Lambda$.
Also, more generally for $m\in\mathbb R$, the operator $\mathrm{Op}(\sigma)$ extends to a bounded operator from the Sobolev space $H^s(\mathbb R^n)$ to $H^{s-m}(\mathbb R^n)$, again for each $\lambda\in\Lambda$.

Another important basic fact concerning pseudodifferential operators is that the composition of pseudodifferential operators is again a pseudodifferential operator, and, modulo lower-order terms, the symbol of the composition is the product of the symbols of the original operators. More precisely, if we have $\sigma_1\in S^{m_1}_{1,0,k}(\mathbb R^n\times\mathbb R^n;\Lambda)$ and $\sigma_2\in S^{m_2}_{1,0,k}(\mathbb R^n\times\mathbb R^n;\Lambda)$ for some $m_1,m_2\in\mathbb R$, then
\[\mathrm{Op}(\sigma_1)\,\mathrm{Op}(\sigma_2)-\mathrm{Op}(\sigma_1\sigma_2)
\in\mathrm{Op}\!\left[S^{m_1+m_2-1}_{1,0,k}(\mathbb R^n\times\mathbb R^n;\Lambda)\right].\]

\subsection{Estimates for the Born approximation}

We will consider large negative energies. But since it is easier to consider positive reals, we shall write $P+\lambda$ with $\lambda>0$ instead of $P-\lambda$ with $\lambda<0$.

\begin{lemma}\label{estimates-for-u}
Let $u\in H_0^k(\Omega)$, $f\in L^2(\Omega)$, and $\lambda\in\mathbb R_+$, and assume that
\[\left(P+\lambda\right)u=f\]
in $\Omega$. Then, for $\lambda\gg1$,
\[\sum_{\left|\alpha\right|\leqslant k}\lambda^{-\left|\alpha\right|/k}\bigl\|\partial^\alpha u\bigr\|
\ll\frac1\lambda\bigl\|f\bigr\|.\]
\end{lemma}

\begin{proof}
We extend $u$ and $f$ by zero extensions to $\mathbb R^n$, so that $u\in H^k(\mathbb R^n)$ and $f\in L^2(\mathbb R^n)$. Also, now
\[\left(P+\lambda\right)u=f\]
in $\mathbb R^n$. Taking Fourier transforms, we have
\[\bigl(P(\cdot)+\lambda\bigr)\widehat u=\widehat f\]
in $\mathbb R^n$, and we immediately get
\begin{align*}
\sum_{\left|\alpha\right|\leqslant k}\lambda^{-\left|\alpha\right|/k}\bigl\|\partial^\alpha u\bigr\|
&\asymp\bigl\|\bigl(\lambda^{-1/k}\left|\cdot\right|+1\bigr)^k\widehat u\bigr\|
\asymp\frac1\lambda\bigl\|\bigl(P(\cdot)+\lambda\bigr)\widehat u\bigr\|
=\frac1\lambda\bigl\|\widehat f\,\bigr\|
=\frac1\lambda\bigl\|f\bigr\|.
\end{align*}
\end{proof}

\begin{lemma}\label{estimates-for-w}
Let $w\in H^k_{P}(\Omega)$, $g\in L^2(\Omega)$, $\lambda\in\mathbb R_+$, and let us fix some cut-off function $\chi\in C_{\mathrm c}^\infty(\Omega)$. Assume that
\[\left(P+\lambda\right)w=g\]
in $\Omega$. Then, for $\lambda\gg1$,
\[\bigl\|\chi w\bigr\|\ll\lambda^{-1/k}\bigl\|w\bigr\|+\lambda^{-1}\bigl\|g\bigr\|,\]
and for multi-indices $\alpha$ with $\left|\alpha\right|\leqslant k$,
\[\bigl\|\partial^\alpha(\chi w)\bigr\|\ll\bigl\|w\bigr\|+\lambda^{\left|\alpha\right|/k-1}\bigl\|g\bigr\|.\]
\end{lemma}

\begin{proof}
We choose and fix cut-off functions $\chi_1,\chi_2\in C_{\mathrm c}^\infty(\Omega)$ so that $\chi_1\equiv1$ in $\mathrm{supp}\,\chi$ and $\chi_2\equiv1$ in $\mathrm{supp}\,\chi_1$. We write
\[B=\mathrm{Op}\!\left(\frac\chi{P(\cdot)+\lambda}\right),\]
so that
\[B\left(P+\lambda\right)=\chi,\]
and $B\in\mathrm{Op}\bigl[S_{1,0,k}^{-k}\bigr]$, and $\bigl\|B\bigr\|\ll\lambda^{-1}$. By pseudodifferential calculus, we have
\[B\chi_1\left(P+\lambda\right)=\chi+K,\]
where $K\in\mathrm{Op}\bigl[S_{1,0,k}^{-1}\bigr]$, and so $\bigl\|K\bigr\|\ll\lambda^{-1/k}$. In particular, we now have
\[B\chi_1\left(P+\lambda\right)\chi_2=\chi+K\chi_2.\]

Since
\[\chi_1\left(P+\lambda\right)\chi_2w=\chi_1g,\]
we now have
\[B\chi_1\left(P+\lambda\right)\chi_2w=B\chi_1g,\]
and so
\[\chi w=-K\chi_2w+B\chi_1g.\]
We immediately get the estimates
\[\bigl\|\chi w\bigr\|\ll\lambda^{-1/k}\bigl\|w\bigr\|+\lambda^{-1}\bigl\|g\bigr\|,\]
and
\[\bigl\|\partial^\alpha(\chi w)\bigr\|\ll\bigl\|w\bigr\|+\lambda^{1/k-1}\bigl\|g\bigr\|,\]
for multi-indices $\alpha$ with $\left|\alpha\right|=1$.

Next, suppose that $N\in\left\{1,2\ldots,k-1\right\}$ is such that we have the estimates
\[\bigl\|\partial^\alpha(\chi w)\bigr\|\ll\bigl\|w\bigr\|+\lambda^{\left|\alpha\right|/k-1}\bigl\|g\bigr\|\]
for all multi-indices $\alpha$ with $\left|\alpha\right|\leqslant N$, and any fixed cut-off function $\chi$. Then, for a multi-index $\alpha$ with $\left|\alpha\right|=N+1$, we shall write $\alpha=\beta+\gamma$ with multi-indices $\beta$ and $\gamma$ such that $\left|\beta\right|=N$ and $\left|\gamma\right|=1$. Arguing as above, we arrive at an identity
\[\chi w=-K\chi_2w+B\chi_1g.\]
To estimate $\bigl\|\partial^\alpha(\chi w)\bigr\|$ we first observe that
\[\bigl\|\partial^\alpha B\chi_1g\bigr\|\ll\lambda^{(N+1)/k-1}\bigl\|g\bigr\|,\]
because $\partial^\alpha B\in\mathrm{Op}\bigl[S^{N+1-k}_{1,0,k}\bigr]$. Next, we observe, that by the induction hypothesis, $\chi_2w\in H^N(\mathbb R^n)$ with
\[\bigl\|\chi_2w\bigr\|_{H^N(\mathbb R^n)}\ll\bigl\|w\bigr\|+\lambda^{N/k-1}\bigl\|g\bigr\|,\]
and since $\partial^\gamma K\in\mathrm{Op}\bigl[S^0_{1,0,k}\bigr]$, we have
\[\bigl\|\partial^\gamma K\chi_2w\bigr\|_{H^N(\mathbb R^n)}\ll\bigl\|w\bigr\|+\lambda^{N/k-1}\bigl\|g\bigr\|,\]
and since
\[\bigl\|\partial^\beta\partial^\gamma K\chi_2w\bigr\|\ll\bigl\|\partial^\gamma K(\chi_2w)\bigr\|_{H^N(\mathbb R^n)},\]
we may invoke induction to finish the proof.
\end{proof}

\begin{lemma}\label{resolvent-estimates}
Let $V\in L^\infty(\Omega)$ be such that in some neighbourhood of $\partial\Omega$ it only takes values from some closed complex half-plane not containing zero. Also, let $W_\alpha\in L^\infty_{\mathrm c}(\Omega)$ be complex-valued for each multi-index $\alpha$ with $1\leqslant\left|\alpha\right|\leqslant e$, where $e\in\left\{0,1,\ldots,k-1\right\}$, and write $W$ for the partial differential operator
\[W=\sum_{1\leqslant\left|\alpha\right|\leqslant e}W_\alpha\partial^\alpha.\]
In the case $e=0$ we set $W=0$.
Let $u\in H_0^k(\Omega)$ and $w\in H_{P}^k(\Omega)$ solve the system
\[\left\{\!\!\begin{array}{l}
\left(P+\mu W+\lambda\right)u+\lambda^{-\nu}Ww+Vw=f,\\[2pt]
\left(P+\iota W+\lambda\right)w=g,
\end{array}\right.\]
in $\Omega$ with $f,g\in L^2(\Omega)$, $\nu,\mu,\iota\in\left\{0,1\right\}$ and $\lambda\in\mathbb R_+$. Then, for $\lambda\gg1$,
\[\sum_{\left|\alpha\right|\leqslant k}\lambda^{-\left|\alpha\right|/k}\bigl\|\partial^\alpha u\bigr\|\ll\frac1\lambda\bigl\|f\bigr\|+\lambda^{\varepsilon+e/k-2}\bigl\|g\bigr\|,\]
and
\[\bigl\|w\bigr\|\ll\bigl\|f\bigr\|+\lambda^{\varepsilon+e/k-1}\bigl\|g\bigr\|.\]
Given a fixed cut-off function $\chi\in C_{\mathrm c}^\infty(\Omega)$, we have, in addition, the local estimates
\[\bigl\|\chi w\bigr\|\ll\lambda^{-1/k}\bigl\|f\bigr\|+\left(\lambda^{\varepsilon+(e-1)/k-1}+\lambda^{-1}\right)\bigl\|g\bigr\|,\]
and
\[
\bigl\|\partial^\alpha(\chi w)\bigr\|\ll\bigl\|f\bigr\|+\left(\lambda^{\varepsilon+e/k-1}+\lambda^{\left|\alpha\right|/k-1}\right)\bigl\|g\bigr\|\]
for multi-indices $\alpha$ with $1\leqslant\left|\alpha\right|\leqslant k$,
and the non-local estimate
\[\bigl\|Pw\bigr\|\ll\lambda\bigl\|f\bigr\|+\lambda^{\varepsilon+e/k}\bigl\|g\bigr\|,\]
both again for $\lambda\gg1$.
\end{lemma}

\begin{proof}
We consider first the case $\mu=\iota=0$.
Multiplying the first equation by $\overline w$, taking complex conjugates of the second equation and multiplying by $u$, and integrating over $\Omega$ and subtracting, we get
\[\int V\bigl|w\bigr|^2=\int\overline wf-\int u\overline g-\lambda^{-\nu}\int\overline wWw,\]
where we used the fact that $\int\overline wPu=\int\overline{Pw}u$.
We pick a cut-off function $\chi\in C_{\mathrm c}^\infty(\Omega)$ so that $V$ only takes values from some closed complex half-plane not containing zero in $\mathrm{supp}\left(1-\chi\right)$, and such that $\chi\equiv1$ in all the supports of all the coefficients of $W$. Now we may estimate using the Cauchy--Schwarz inequality and Lemma \ref{estimates-for-u}, for arbitrarily small $\kappa\in\mathbb R_+$,
\begin{align*}
&\bigl\|\left(1-\chi\right)w\bigr\|^2\ll\int V\bigl|w\bigr|^2
\ll\bigl\|w\bigr\|\cdot\bigl\|f\bigr\|+\bigl\|g\bigr\|\cdot\bigl\|u\bigr\|+\bigl\|\chi w\bigr\|\cdot\lambda^{-\nu}\bigl\|Ww\bigr\|\\
&\ll\frac1\kappa\bigl\|f\bigr\|^2+\kappa\bigl\|w\bigr\|^2+\bigl\|g\bigr\|\left(\frac1\lambda\bigl\|f\bigr\|+\frac1\lambda\bigl\|w\bigr\|+\frac1\lambda\bigl\|Ww\bigr\|\right)+\lambda^{\varepsilon}\bigl\|\chi w\bigr\|^2+\frac{\bigl\|Ww\bigr\|^2}{\lambda^{2\nu+\varepsilon}}\\
&\ll\frac1\kappa\bigl\|f\bigr\|^2+\kappa\bigl\|\left(1-\chi\right)w\bigr\|^2+\lambda^{\varepsilon}\bigl\|\chi w\bigr\|^2+\lambda^{\varepsilon-2}\bigl\|g\bigr\|^2+\lambda^{-\varepsilon}\bigl\|Ww\bigr\|^2.
\end{align*}
The norms involving $\chi w$ and $Ww$ can be estimated by Lemma \ref{estimates-for-w} as
 \[\lambda^\varepsilon\bigl\|\chi w\bigr\|+\lambda^{-\varepsilon}\bigl\|Ww\bigr\|
\ll\lambda^{-\varepsilon}\bigl\|w\bigr\|+\lambda^{e/k-1}\bigl\|g\bigr\|.\]
Thus, for $\lambda\gg1$ and sufficiently small fixed $\kappa$, the norm of $\left(1-\chi\right)w$ on the right-hand side can be absorbed to the left-hand side, and we may continue by Lemma \ref{estimates-for-w}
\begin{align*}
\bigl\|w\bigr\|&\ll
\bigl\|\left(1-\chi\right)w\bigr\|+\bigl\|\chi w\bigr\|\ll\bigl\|f\bigr\|+\lambda^{\varepsilon-1}\bigl\|g\bigr\|+\lambda^\varepsilon\bigl\|\chi w\bigr\|+\lambda^{-\varepsilon}\bigl\|Ww\bigr\|\\
&\ll\bigl\|f\bigr\|+\lambda^{\varepsilon-1}\bigl\|g\bigr\|+\lambda^{-\varepsilon}\bigl\|w\bigr\|+\lambda^{e/k-1}\bigl\|g\bigr\|.
\end{align*}
For $\lambda\gg1$, the norm of $w$ can be absorbed to the left-hand side giving
\[\bigl\|w\bigr\|\ll\bigl\|f\bigr\|+\lambda^{\varepsilon+e/k-1}\bigl\|g\bigr\|.\]
Combining the above estimates with Lemmas \ref{estimates-for-u} and \ref{estimates-for-w} gives
\[\sum_{\left|\alpha\right|\leqslant k}\lambda^{-\left|\alpha\right|/k}\bigl\|\partial^\alpha u\bigr\|\ll\frac1\lambda\bigl\|f\bigr\|+\frac1\lambda\bigl\|Vw\bigr\|+\frac1\lambda\bigl\|Ww\bigr\|
\ll\frac1\lambda\bigl\|f\bigr\|+\lambda^{\varepsilon+e/k-2}\bigl\|g\bigr\|,\]
as desired.

The estimates for $\chi w$ now follow from combining the above estimates with Lemma \ref{estimates-for-w}, and the estimate for $Pw$ follows immediately from the equation satisfied by $w$ and the previous estimates for $w$. The case $\mu=1$, $\nu=0$ follows immediately from the case $\mu=\nu=0$ with $f$ replaced by $f-Wu$, as the case $\mu=0$ then implies for the case $\mu=1$ that
\[\bigl\|Wu\bigr\|\ll\lambda^{e/k-1}\bigl\|f\bigr\|+\lambda^{e/k-1}\bigl\|Wu\bigr\|+\lambda^{\varepsilon+2e/k-2}\bigl\|g\bigr\|,\]
where the term $\left\|Wu\right\|$ on the right-hand side can be absorbed to the left-hand side for sufficiently large $\lambda$, and now this estimate for $Wu$ can be used to get the required estimates from those given by the case $\mu=0$.
Finally, whether $\mu=0$ or $\mu=1$, the case $\iota=1$ follows from the case $\iota=0$ in the same fashion, except now we replace $g$ by $g-Ww$, and the case $\iota=0$ gives the estimate
\[\bigl\|Ww\bigr\|\ll\bigl\|f\bigr\|+\lambda^{\varepsilon+e/k-1}\bigl\|g\bigr\|,\]
which can be used to transform the estimates from the case $\iota=0$ to those required in the case $\iota=1$.
\end{proof}

\subsection{Resolvent operators of the Born approximation}\label{roba}

The following is a special case of Banach's closed range theorem (see e.g. Sect. VII.5 in \cite{Yosida}).
\begin{theorem}\label{banach-closed-range-theorem}
Let $H$ be a Hilbert space, and let $T$ be a closed densely defined operator of $H$. Then the following four statements are equivalent:
\begin{enumerate}\setlength{\itemsep}{0pt}
\item $\mathrm{Im}\,T$ is closed in $H$.
\item $\mathrm{Im}\,T^\ast$ is closed in $H$.
\item $\mathrm{Im}\,T=(\mathrm{Ker}\,T^\ast)^\perp$.
\item $\mathrm{Im}\,T^\ast=(\mathrm{Ker}\,T)^\perp$.
\end{enumerate}
In particular, to show that $T$ is surjective, it is enough to show that $\mathrm{Im}\,T$ is closed and that $T^\ast$ is injective.
\end{theorem}

Following \cite{Sylvester}, we first consider the Born approximation, the operator matrix
\[B=\begin{bmatrix}
P+\nu W&\lambda^{-\nu}W+V\\
0&P
\end{bmatrix}\colon H_0^k(\Omega)\times H^k_{P}(\Omega)\longrightarrow L^2(\Omega)\times L^2(\Omega).\]
More precisely, we will consider $B$ at large negative energies. The relevant properties will be as follows:
\begin{proposition}\label{resolvent-of-the-born-approximation}
Let $P$ be an elliptic formally self-adjoint constant coefficient partial differential operator of order $k\in\mathbb Z_+$ in $\mathbb R^n$ bounded from below, let $\Omega$ be an open nonempty bounded subset of $\mathbb R^n$, and let $V\in L^\infty(\Omega)$ be such that in some neighbourhood of $\partial\Omega$ it only takes values from some closed complex half-plane not containing zero. Also, let $W_\alpha\in W^{\left|\alpha\right|,\infty}_{\mathrm c}(\Omega)$ for each multi-index $\alpha$ with $1\leqslant\left|\alpha\right|\leqslant e$, where $e\in\left\{0,1,\ldots,k-1\right\}$, and write $W$ for the partial differential operator
\[W=\sum_{1\leqslant\left|\alpha\right|\leqslant e}W_\alpha\partial^\alpha.\]
In the case $e=0$ we set $W=0$.

Now, let $\nu,\mu\in\left\{0,1\right\}$ and let $B$ denote the operator
\[\begin{bmatrix}P+\mu W&\lambda^{-\nu}W+V\\0&P\end{bmatrix}\colon
H^k_0(\Omega)\times H^k_P(\Omega)\longrightarrow L^2(\Omega)\times L^2(\Omega).\]
Then, for sufficiently large $\lambda\in\mathbb R_+$, the operator $B+\lambda$ is bijective, and we may consider the resolvent operator matrix
\[(B+\lambda)^{-1}=\begin{bmatrix}
R_{11}&R_{12}\\
R_{21}&R_{22}
\end{bmatrix}\colon L^2(\Omega)\times L^2(\Omega)\longrightarrow L^2(\Omega)\times L^2(\Omega).\]
Here the operators $R_{11}$, $R_{12}$ and $R_{22}$ are compact, $R_{21}$ is bounded, and $\chi R_{21}$ is compact for any $\chi\in C_{\mathrm c}^\infty(\Omega)$. The operators $\partial^\alpha R_{11}$ and $\partial^\alpha R_{12}$ are compact for multi-indices $\alpha$ with $1\leqslant\left|\alpha\right|<k$, as are $\partial^\alpha\chi R_{21}$ and $\partial^\alpha\chi R_{22}$. Furthermore, we have the operator norm bounds
\[\bigl\|\partial^\alpha R_{11}\bigr\|\ll\lambda^{\left|\alpha\right|/k-1},\quad
\bigl\|\partial^\alpha R_{12}\bigr\|\ll\lambda^{\varepsilon+\left|\alpha\right|/k-2},\quad\text{and}\quad
\bigl\|R_{22}\bigr\|\ll\lambda^{\varepsilon+e/k-1},\]
and
\[\bigl\|R_{21}\bigr\|\ll1,\quad\text{and}\quad\bigl\|\chi R_{21}\bigr\|\ll\lambda^{-1/k},\]
where $\chi\in C_{\mathrm c}^\infty(\Omega)$ is fixed, and $\alpha$ is a multi-index with $\left|\alpha\right|\leqslant k$.
\end{proposition}

\begin{proof}
Lemma \ref{resolvent-estimates} immediately implies that $B+\lambda$ is injective for sufficiently large real $\lambda$. The image of $B+\lambda$ turns out to be closed. Namely, let $\left\langle f_n,g_n\right\rangle_{n=1}^\infty\subseteq\mathrm{Im}\,(B+\lambda)$ satisfying $f_n\longrightarrow f$ and $g_n\longrightarrow g$ for some $f,g\in L^2(\Omega)$ when $n\longrightarrow\infty$, and let $\left\langle u_n,w_n\right\rangle_{n=1}^\infty\subseteq H^k_0(\Omega)\times H^k_P(\Omega)$ be such that $(B+\lambda)\left\langle u_n,w_n\right\rangle=\left\langle f_n,g_n\right\rangle$ for each $n\in\mathbb Z_+$.
Since the sequence $\left\langle f_n,g_n\right\rangle_{n=1}^\infty$ is Cauchy in $L^2(\Omega)\times L^2(\Omega)$, Lemma \ref{resolvent-estimates} implies that the sequence $\left\langle u_n,w_n\right\rangle_{n=1}^\infty$ is Cauchy in $H^k_0(\Omega)\times H^k_P(\Omega)$ and therefore has a unique limit $\left\langle u,w\right\rangle\in H^k_0(\Omega)\times H^k_P(\Omega)$, which by Lemma \ref{resolvent-estimates} must satisfy $(B+\lambda)\left\langle u,w\right\rangle=\left\langle f,g\right\rangle$. Thus $\mathrm{Im}\,(B+\lambda)$ is closed.

Since the adjoint of $B$ is
\[B^\ast=\begin{bmatrix}P+\mu W^\ast&0\\\lambda^{-\nu}W^\ast+\overline V&P\end{bmatrix}\colon H^k_P(\Omega)\times H^k_0(\Omega)\longrightarrow L^2(\Omega)\times L^2(\Omega),\]
and the zeroeth-order term of $\lambda^{-\nu}W^\ast+\overline V$ satisfies the coercivity condition as all the coefficients of $W$ vanish in a neighbourhood of $\partial\Omega$,
Lemma \ref{resolvent-estimates} also implies that $B^\ast+\lambda$ is injective, using the values $\mu=0$ and $\iota=1$, and so, by Theorem \ref{banach-closed-range-theorem}, the operator $B+\lambda$ is surjective. We have now established that for $\lambda\gg1$, the operator $B+\lambda$ is bijective and we may move on to discuss $(B+\lambda)^{-1}$.

The operator $R_{11}$ is the projection to the first component of the restriction of $(B+\lambda)^{-1}$ to the subspace $L^2(\Omega)\times\left\{0\right\}$. In other words, given $f\in L^2(\Omega)$, there exists unique functions $u\in H^k_0(\Omega)$ and $w\in H^k_P(\Omega)$ solving
\[\left\{\!\!\begin{array}{l}
(P+\mu W+\lambda)u+\lambda^{-\nu}Ww+Vw=f,\\
(P+\lambda)w=0,\end{array}\right.\]
and the operator $R_{11}$ is the mapping $f\longmapsto u\colon L^2(\Omega)\longrightarrow L^2(\Omega)$ taking values in $H^k_0(\Omega)$. By Lemma \ref{resolvent-estimates}, we have $\bigl\|\partial^\alpha R_{11}\bigr\|\ll\lambda^{\left|\alpha\right|/k-1}$ for multi-indices $\alpha$ with $\left|\alpha\right|\leqslant k$. Similar arguments combined with Lemma \ref{resolvent-estimates} also give the operator norm bounds
\[\bigl\|\partial^\alpha R_{12}\bigr\|\ll\lambda^{\varepsilon+\left|\alpha\right|/k-2},\quad
\bigl\|R_{21}\bigr\|\ll1,\quad\text{and}\quad
\bigl\|R_{22}\bigr\|\ll\lambda^{\varepsilon+e/k-1},\]
again for $\left|\alpha\right|\leqslant k$.
Similarly, using Lemma \ref{resolvent-estimates} and a similar argument, we get the operator norm bound $\bigl\|\chi R_{21}\bigr\|\ll\lambda^{-1/k}$ for any $\chi\in C_{\mathrm c}^\infty(\Omega)$.

The operators $R_{11}$ and $R_{12}$ are compact, since by Lemmas \ref{estimates-for-u} and \ref{resolvent-estimates}, they map $L^2(\Omega)$ boundedly into $H^k_0(\Omega)$ and the latter embeds compactly into $L^2(\Omega)$. Similarly, Lemma \ref{resolvent-estimates} implies that for $\chi\in C_{\mathrm c}^\infty(\Omega)$, the operators $\chi R_{21}$ and $\chi R_{22}$ map $L^2(\Omega)$ boundedly into $H^1_0(\Omega)$ which in turn embeds compactly into $L^2(\Omega)$. Also, it is clear that $\partial^\alpha R_{11}$ and $\partial^\alpha R_{12}$, which map $L^2(\Omega)$ to $H^{k-\left|\alpha\right|}_0(\Omega)$ boundedly, are compact operators of $L^2(\Omega)$ for $1\leqslant\left|\alpha\right|<k$. We get the same claim for $\partial^\alpha\chi R_{21}$ and $\partial^\alpha\chi R_{22}$ as $\chi R_{21}$ and $\chi R_{22}$ map $L^2(\Omega)$ boundedly to $H^{k-\left|\alpha\right|}_0(\Omega)$.

Finally, it only remains to prove that $R_{22}$ is compact as well. Let $\left\langle g_n\right\rangle_{n=1}^\infty\subseteq L^2(\Omega)$ be a sequence converging weakly to zero, and let, for each $n\in\mathbb Z_+$, the functions
$u_n\in H^k_0(\Omega)$ and $w_n\in H^k_P(\Omega)$ be the unique solution to
\[\left\{\!\!\begin{array}{l}
(P+\mu W+\lambda)u_n+\lambda^{-\nu}Ww_n+Vw_n=0,\\
(P+\lambda)w_n=g_n,\end{array}\right.\]
so that $R_{22}g_n=w_n$. If we can show that $w_n\longrightarrow0$ in $L^2(\Omega)$ as $n\longrightarrow\infty$, then $R_{22}$ must be compact.

Let $K\subseteq\Omega$ be compact, and let $\chi_K$ be its characteristic function. Let us assume that $K$ is so large that $V$ satisfies the coercivity condition in $\Omega\setminus K$ and such that $K$ contains the supports of $W_\alpha$. Also, we assume as we may that $K$ contains an open set which contains a compact set $K'$ such that $V$ satisfies the coercivity condition in $\Omega\setminus K'$ as well. First, pick a cut-off function $\chi_1\in C_{\mathrm c}^\infty(\Omega)$ such that $\chi_1\equiv1$ on $K$. Since $\chi_1R_{22}$ is compact, we know that $\chi_1w_n\longrightarrow0$ as $n\longrightarrow\infty$, so that also $\chi_Kw_n$ tends to zero, and it only remains to prove that $(1-\chi_K)w_n$ tends to zero.

Next, let $\chi_2\in C_{\mathrm c}^\infty(\Omega)$ be such that $\mathrm{supp}\,\chi_2\subseteq K$ and $\chi_2\equiv1$ in $K'$, so that $V$ satisfies the coercivity condition in $\mathrm{supp}\,(1-\chi_2^2)$. The system solved by $u_n$ and $w_n$ implies easily that
\[\int V\left|w_n\right|^2=-\int u_n\overline g_n-\lambda^{-\nu}\int\overline{w_n}Ww_n-\mu\int\overline{w_n}\,Wu_n,\]
so that
\begin{align*}
\int_{\Omega\setminus K}\left|w_n\right|^2&\ll\int V(1-\chi_2^2)\left|w_n\right|^2\\
&=\int V\left|\chi_2w_n\right|^2-\int u_n\overline g_n
-\int\overline{\chi_1w_n}\,Ww_n-\mu\int\overline{\chi_1w_n}Wu_n.
\end{align*}
The first term on the right-hand side tends to zero as $n\longrightarrow\infty$ since $\chi_2R_{22}$ is compact, and the second term tends to zero since $R_{12}$ is compact and the sequence $\left\langle g_n\right\rangle_{n=1}^\infty$ is bounded. In the third term, $\chi_1w_n$ tends to zero as $\chi_1R_{22}$ is compact, and $Ww_n$ is bounded, by Lemma \ref{resolvent-estimates} and the fact that the sequence $\left\langle g_n\right\rangle_{n=1}^\infty$ is bounded. The fourth term, if it exists, tends to zero as $\chi_1R_{22}$ is compact and $\left\|Wu_n\right\|\ll\lambda^{\varepsilon+e/k-2}\left\|g_n\right\|$.
\end{proof}

\subsection{Proofs
of Theorems
\ref{schrodinger-main-theorem} and \ref{helmholtz-main-theorem}
}

The following theorem is a special case of Theorem 3.5 in \cite{Sylvester}.
\begin{theorem}\label{UTC-resolvent-theorem}
Let $H$ be a Hilbert space, and let $B$ be a closed densely defined operator on $H\times H$ with domain $\mathrm{Dom}\,B$. Let us assume that there exists a complex number $\lambda$ such that the operator
\[B-\lambda\colon\mathrm{Dom}\,B\longrightarrow H\times H\]
is bijective, and that the inverse
\[(B-\lambda)^{-1}=\begin{bmatrix}R_{11}&R_{12}\\R_{21}&R_{22}\end{bmatrix}\colon H\times H\longrightarrow H\times H\]
is bounded, and that the components $R_{11}$, $R_{12}$ and $R_{22}$ are compact. Then the spectrum of $B$ consists of a discrete set of eigenvalues with finite-dimensional generalized eigenspaces.
\end{theorem}

\begin{proof}[Proof of Theorem \ref{schrodinger-main-theorem}.]
Let $B$ and $\widetilde B$ denote the operator matrices
\[\begin{bmatrix}P&W+V\\0&P\end{bmatrix}\colon
H^k_0(\Omega)\times H^k_P(\Omega)\longrightarrow L^2(\Omega)\times L^2(\Omega)\]
and
\[\begin{bmatrix}P+W+V&W+V\\0&P\end{bmatrix}\colon
H^k_0(\Omega)\times H^k_P(\Omega)\longrightarrow L^2(\Omega)\times L^2(\Omega),\]
respectively.
The operator $B$ has already been shown in Proposition \ref{resolvent-of-the-born-approximation} to have well behaved resolvent operators. The operators $B$ and $\widetilde B$ are connected by the elementary identity
\begin{align*}
\widetilde B+\lambda&=\left(1+\begin{bmatrix}W+V&0\\0&0\end{bmatrix}(B+\lambda)^{-1}\right)(B+\lambda)\\
&=\left(1+\begin{bmatrix}(W+V)R_{11}&(W+V)R_{12}\\0&0\end{bmatrix}\right)(B+\lambda),
\end{align*}
where $\lambda$ is a sufficiently large positive real number, and $R_{11}$ and $R_{12}$ are as in Proposition \ref{resolvent-of-the-born-approximation}. Since by Proposition \ref{resolvent-of-the-born-approximation}, the operator norm of the matrix involving $W$, $V$, $R_{11}$ and $R_{12}$ is $\ll\lambda^{e/k-1}$, we may invert the first factor on the right-hand side by Neumann series, with the inverse being an operator matrix $\begin{bmatrix}A&C\\0&1\end{bmatrix}$ for some bounded operators $A$ and $C$. It is easy to check that $C=-A\left(W+V\right)R_{12}$, so that $C$ is compact. Now $(\widetilde B+\lambda)^{-1}$ is upper triangular compact and Theorem~\ref{UTC-resolvent-theorem} implies the desired spectral properties for the perturbed operator $\widetilde B$.
\end{proof}

For the proof of the Helmholtz result, the perturbation $\lambda V$ is too large and we cannot invert via Neumann series. Instead we will invert via a Fredholm argument by using the following special case of Proposition 3.4 in \cite{Sylvester}.
\begin{theorem}\label{sylvester34}
Let $H$ be a Hilbert space, let $D\subseteq\mathbb C$ be a connected open set, and let
\[A(\lambda)=\begin{bmatrix}A_{11}(\lambda)&A_{12}(\lambda)\\A_{21}(\lambda)&A_{22}(\lambda)\end{bmatrix}\colon H\times H\longrightarrow H\times H\]
be a bounded operator for each $\lambda\in D$ with $A_{11}(\lambda)$, $A_{12}(\lambda)$ and $A_{22}(\lambda)$ compact. Assume that $A(\lambda)$ depends analytically on $\lambda$. Let $\lambda_0\in D$ be such that $\mathrm{Ker}\left(1+A(\lambda_0)\right)=\left\{0\right\}$ or $\mathrm{CoKer}\left(1+A(\lambda_0)\right)=\left\{0\right\}$. Then $1+A(\lambda_0)$ is bijective with a bounded inverse, and so is $1+A(\lambda)$ for all but a discrete set of $\lambda\in D$. At the exceptional points the kernel and cokernel of $1+A(\lambda)$ are finite dimensional.
\end{theorem}

\begin{proof}[Proof of Theorem \ref{helmholtz-main-theorem}.]
The proof in the Helmholtz case follows the same lines, except that now we will consider the system for $u$ and $\widetilde w=w/\lambda$, leading to the operator
\[\widetilde B(\lambda)=\begin{bmatrix}
P+W+\lambda V&W/\lambda+V\\
0&P
\end{bmatrix}\colon H^k_0(\Omega)\times H^k_P(\Omega)\longrightarrow L^2(\Omega)\times L^2(\Omega),\]
and the corresponding Born approximation
\[B(\lambda)=\begin{bmatrix}
P+W&W/\lambda+V\\
0&P
\end{bmatrix}\colon H^k_0(\Omega)\times H^k_P(\Omega)\longrightarrow L^2(\Omega)\times L^2(\Omega).\]
It is clear that this $B(\lambda)$ has the same excellent properties as in the Schr\"odinger case, and that the same estimates apply.
The main difference from the situation of Theorem \ref{schrodinger-main-theorem} is that in this case the perturbation that needs to be added to $B(\lambda)$ involves $\lambda V$ instead of $V$, and unfortunately this is a rather large perturbation, and we no longer can invert via Neumann series. Instead, we first observe that, given $\lambda_0\in\mathbb R_+$ that is so large that $B(\lambda_0)+\lambda_0$ is invertible with a bounded upper triangular compact inverse, we have
\[B(\lambda)+\lambda=\left(1+\begin{bmatrix}\lambda-\lambda_0&W/\lambda-W/\lambda_0\\0&\lambda-\lambda_0\end{bmatrix}\left(B(\lambda_0)+\lambda_0\right)^{-1}\right)\left(B(\lambda_0)+\lambda_0\right),\]
for $\lambda\in\mathbb C^\times$. When $\lambda=\lambda_0$, the left-hand side is surjective and the first term on the right-hand side is therefore invertible by Theorem \ref{sylvester34}, and $\left(B(\lambda)+\lambda\right)^{-1}$ exists for all $\lambda\in\mathbb C^\times$ except possibly for a discrete set of exceptions. In the discrete set of exceptions, the kernel of $B(\lambda)+\lambda$ has a finite dimensional kernel and a finite dimensional cokernel. Furthermore, $\left(B(\lambda)+\lambda\right)^{-1}$ is upper triangular compact when it exists; this follows from the upper triangular compactness of $\left(B(\lambda_0)+\lambda_0\right)^{-1}$ and the fact that the inverse of the first factor $(1+\ldots)$ must be of the form $\begin{bmatrix}\ast&D\\\ast&\ast\end{bmatrix}$, where $D$ is compact and the remaining elements are bounded, as is easy to check.

Next, for those $\lambda\in\mathbb C^\times$, for which $(B(\lambda)+\lambda)$ is invertible as above, we have
\[\widetilde B(\lambda)+\lambda=\left(B(\lambda)+\lambda\right)\left(1+\left(B(\lambda)+\lambda\right)^{-1}\begin{bmatrix}\lambda V&0\\0&0\end{bmatrix}\right).\]
If we prove that $\widetilde B(\lambda)+\lambda$ is injective for large $\lambda\in\mathbb R_+$, then so will be the last factor on the right-hand side. Then Theorem \ref{sylvester34} implies that the factor on the right is invertible with the inverse being a bounded operator matrix $\begin{bmatrix}A&0\\C&1\end{bmatrix}$, and so $\widetilde B(\lambda)+\lambda$ must be invertible with an upper triangular compact inverse no less, except possibly for a discrete set of exceptions. At these possible exceptional points, the kernel of $\widetilde B(\lambda)+\lambda$ must be finite dimensional because the kernel of the factor $1+\ldots$ is. In the case $W=0$ we get the discreteness in $\mathbb C$ simply because the problematic terms $W/\lambda$ never appear.

Thus it is enough to prove that the kernel of $\widetilde B(\lambda)+\lambda$ is merely the trivial $\left\{0\right\}$ for large enough $\lambda\in\mathbb R_+$.
So, we have to prove that if $u\in H^k_0(\Omega)$ and $w\in H^k_P(\Omega)$ solve
\begin{equation*}\label{helmholtz-kernel-system}\begin{cases}
\left(P+\lambda\right)u+Ww/\lambda+Vw=-Wu-\lambda Vu,\\
\left(P+\lambda\right)w=0,
\end{cases}\end{equation*}
then we must have $u\equiv w\equiv0$. We will assume that $u\not\equiv0$ and derive a contradiction. To simplify notation, we may normalize $\left\|u\right\|=1$.
We immediately get from these equations the estimates
\[\left\|\chi w\right\|\ll\lambda^{-1/k}\left\|Wu\right\|+\lambda^{-1/k}\cdot\lambda\left\|u\right\|\ll\lambda^{-1/k+e/k}+\lambda^{-1/k+1}\ll\lambda^{1-1/k},\]
and
\[\left\|Wu\right\|\ll\lambda^{e/k},\quad\left\|w\right\|\ll\lambda\quad\text{and}\quad\left\|Ww\right\|\ll\lambda.
\]

From the above pair of equations, integrating by parts, it is easy to derive the identities
\[\int\overline u\,Pu+\lambda\int\left|u\right|^2+\frac1\lambda\int u\,\overline{Ww}+\int u\,\overline{Vw}=-\int u\,\overline{Wu}-\lambda\int\overline V\left|u\right|^2,\]
and
\[\int\overline w\,(P+\lambda)u+\frac1\lambda\int\overline w\,Ww+\int V\left|w\right|^2=-\int\overline w\,Wu-\lambda\int\overline w\,Vu,\]
as well as the simple identity
\[\int\overline w\,(P+\lambda)u=0.\]
Furthermore, it is easy to see that
\[\int u\,\overline{Vw}-\int u\,V\,\overline w
\ll\left\|\Im V\right\|_{L^\infty}\cdot\left\|w\right\|
\ll\lambda\left\|\Im V\right\|_{L^\infty}.\]
Also, using the easy fact that $W$ is $P$-form infinitesimal (see Lemma \ref{form-infinitesimality} below), we get
\[\Re\int u\,\overline{Wu}>-\kappa\int\overline u\,Pu-C_\kappa\int\left|u\right|^2,\]
where $\kappa$ is an arbitrarily small positive real number and $C_\kappa$ is some positive real number depending on $\kappa$ (and $n$, $\Omega$, $P$ and $W$).

Finally, let us introduce one more object: let $\chi\in C_{\mathrm c}^\infty(\Omega)$ be such that it takes values only from $\left[0,1\right]$, such that it is $\equiv1$ in some compact set $K\subseteq\Omega$ such that $V$ satisfies the coercivity relation in $\Omega\setminus K$ and $K$ contains the supports of~$W_\alpha$.

Now, the above observations lead to the complicated identity
\begin{multline*}
\int\overline u\,Pu+\lambda\int\left|u\right|^2+\lambda\int\overline V\left|u\right|^2+\int u\,\overline{Wu}+O\bigl(\lambda\left\|\Im V\right\|_{L^\infty}\bigr)+O(1)\\
=\frac1\lambda\int V\left(1-\chi^2\right)\left|w\right|^2
+\frac1\lambda\int V\left|\chi w\right|^2+\frac1\lambda\int\overline w\,Wu
+\frac1{\lambda^2}\int\overline w\,Ww.
\end{multline*}
The real part of the left-hand side is
\[\geqslant(1-\kappa)\int\overline u\,Pu+\Re\int\left(\lambda(1+\overline V)-C_\kappa\right)\left|u\right|^2+O\bigl(\lambda\left\|\Im V\right\|_{L^\infty}\bigr)+O(1),\]
and this is certainly positive and $\gg\lambda$ as $\lambda\longrightarrow\infty$, provided that $\left\|\Im V\right\|_{L^\infty}$ was indeed sufficiently small. The positivity follows from the fact that $u$ cannot vanish identically near $\Omega\setminus K$, for otherwise $w$ would also vanish in $\Omega\setminus K$, and a fortiori vanish in all of $\Omega$. Thus, $u$ would solve $\left(P+\lambda\right)u+Wu+\lambda Vu=0$ in $\mathbb R^n$ and $u\equiv0$ in $\mathbb R^n\setminus K$, implying that $u\equiv0$ in all of $\mathbb R^n$ by the weak unique continuation theorems in \cite{Wang}. (We apply Theorem 1.3 or 1.4 of \cite{Wang} depending on whether $e>0$ or $e=0$.)

The right-hand side, on the other hand, is
\[\frac1\lambda\int V\left(1-\chi^2\right)\left|w\right|^2
+O(\lambda^{1-2/k})
,\]
and due to the coercivity condition, the first term always has a non-positive real part. Thus the right-hand side always has a negative or $o(\lambda)$ real part, and so we have reached the desired contradiction.
\end{proof}

\begin{lemma}\label{form-infinitesimality}
Let $W$ and $P$ be as in Theorem \ref{schrodinger-main-theorem}. Then, for any $\kappa\in\mathbb R_+$ there exists a constant $C_\kappa\in\mathbb R_+$ such that
\[\left|\left\langle u\middle|Wu\right\rangle\right|\leqslant\kappa\left\langle u\middle|Pu\right\rangle+C_\kappa\left\|u\right\|^2\]
for all $u\in H^k(\mathbb R^n)$, where $\left\langle u\middle|v\right\rangle$ denotes $\int_{\mathbb R^n}\overline uv$ for $u,v\in L^2(\mathbb R^n)$.
\end{lemma}

\begin{proof}
The conditions on the coefficients of $W$ guarantee that $W+W^\ast$ and $iW-iW^\ast$ are well-defined and symmetric on $H^k(\mathbb R^n)$, and
it is easy to check by taking Fourier transforms that for any $\kappa\in\mathbb R_+$, there exists a constant $C_\kappa'\in\mathbb R_+$ such that
\[\left\|(W\pm W^\ast)u\right\|^2\leqslant\kappa\left\|Pu\right\|^2+C_\kappa'\left\|u\right\|^2\]
for all $u\in H^k(\mathbb R^n)$. Finally, since
\[\left|\left\langle u\middle|Wu\right\rangle\right|\leqslant
\frac12\left|\left\langle u\middle|(W+W^\ast)u\right\rangle\right|
+\frac12\left|\left\langle u\middle|(iW-iW^\ast)u\right\rangle\right|,\]
the result now follows from Theorem VI.1.38 of \cite{Kato}.
\end{proof}

\end{document}